%% file: modelclass.tex
\documentclass{amsart}

\usepackage{amsmath}
\usepackage{amssymb}
\usepackage{bbm}
\usepackage{enumerate}
\usepackage{appendix}
\usepackage{mathrsfs}

\numberwithin{equation}{section}

\newcommand{\R}{\mathbb{R}}
\newcommand{\N}{\mathbb{N}}

\newcommand{\E}{\mathbb{E}}
\newcommand{\pr}{\mathbb{P}}

\newcommand{\e}{\operatorname{e}}
\newcommand{\dd}{\,{\mathrm d}}
\newcommand{\db}{{\mathrm d}}
\newcommand{\dual}{^{\ast}}

\newcommand{\lo}{\mathcal{L}}
\newcommand{\ham}{\mathcal{H}}
\newcommand{\eps}{\varepsilon}

\newcommand{\pt}{\partial}

\newtheorem{lemma}{Lemma}[section]

\newtheorem{thm}[lemma]{Theorem}

\newtheorem{remark0}[lemma]{Remark}
\newtheorem{eg0}[lemma]{Example}

\newenvironment{remark}{\begin{remark0}\rm}{\hspace*{\fill} $\square$
                        \end{remark0}}

\author[K. Habermann]{Karen Habermann}
\address{University of Bonn, Hausdorff Center for Mathematics,
  Endenicher Allee 62, 53115 Bonn, Germany.}
\email{habermann@iam.uni-bonn.de}
\thanks{Research supported by the German Research Foundation DFG
  through the Hausdorff Center for Mathematics.}

\title[Small-time fluctuations for a model class of hypoelliptic
diffusion bridges]
{Small-time fluctuations for the bridge in a model class of
  hypoelliptic diffusions of weak H\"ormander type}
\begin{document}
\begin{abstract}
  We study the small-time asymptotics for hypoelliptic diffusion
  processes conditioned by their initial and final positions, in a
  model class of diffusions
  satisfying a weak H\"ormander condition
  where the diffusivity is constant
  and the drift is linear. We show that, while the diffusion bridge can
  exhibit a blow-up behaviour in the small time limit, we can still
  make sense of suitably rescaled fluctuations which converge
  weakly. We explicitly describe the limit fluctuation process in
  terms of quantities associated to the unconditioned diffusion.
  In the discussion of examples, we also
  find an expression for the bridge from $0$ to $0$ in time $1$ of an
  iterated Kolmogorov diffusion.
\end{abstract}

\maketitle
\thispagestyle{empty}

\section{Introduction}
\input{introduction}

\section{Diffusion bridge in the model class}
\label{sec:modelbridge}
\input{bridge}

\section{Small-time analysis for the model
  diffusion bridge}
\label{sec:smalltime}
\input{smalltime}

\section{Illustrating examples}
\label{sec:examples}
\input{examples}
\input{diffmin}
\input{genkol}
\bibliographystyle{plain}
\bibliography{references}

\end{document}

%% file: introduction.tex
The small-time asymptotics for hypoelliptic diffusion processes can
depend crucially on the drift term. For instance, Ben Arous and
L\'eandre~\cite{expdecay1,expdecay2} showed that an interaction of the
flow of the drift vector field with the heat diffusion can lead to an
exponential decay of the heat kernel on the diagonal. The current
paper discusses and illustrates the effects the drift term can have on the
small-time fluctuations for hypoelliptic diffusion bridges.

Bailleul, Mesnager and Norris~\cite{BMN} studied the
small-time asymptotics of sub-Riemannian diffusion bridges outside the
cut locus.
Their analysis was extended by us to the diagonal, cf.~\cite{MAP}, to
describe the asymptotics of sub-Riemannian diffusion loops.
Both works are concerned with hypoelliptic diffusion processes
whose associated generators satisfy the so-called strong H\"ormander
condition and where the drift vector fields are nice enough to not
affect the small-time asymptotics. In continuation of this work, we
would like to analyse the small-time asymptotics for
hypoelliptic
diffusion bridges, where one assumes a weak H\"ormander condition
only. As a first step towards this goal, we determine the
small-time bridge fluctuations
for a model class of hypoelliptic diffusions
satisfying a weak H\"ormander condition, and we contrast our results with
\cite{BMN} and \cite{MAP}.

We consider the same model class for which Barilari and
Paoli~\cite{barilariandpaoli} describe the small-time heat kernel
expansion on the diagonal and give a geometric characterisation of the
coefficients in terms of curvature-like invariants. The corresponding
model class of hypoelliptic operators
already features in
the pioneering work of H\"ormander~\cite{hormander}, and
Lanconelli and Polidoro~\cite{evolution} study a notion of principal
part as well as the invariance with respect to suitable
groups of translations and dilations for this class of operators.

Fix $d,m\in\N$.
Let $A$ be a $d\times d$ matrix and $B$ be a $d\times m$ matrix
such that there exists $N\in\N$ with
\begin{equation}\label{kalman}
  \operatorname{rank}\left[B,AB,A^2B,\dots,A^{N-1}B\right]=d\;,
\end{equation}
where
$[B,AB,A^2B,\dots,A^{N-1}B]$ is the matrix formed of the columns
of $B,AB,A^2B,\dots,A^{N-1}B$.
Let $n$ denote the minimal $N$ satisfying~(\ref{kalman}).
We study the diffusion process whose
generator $\lo$ is the second order differential operator
on $\R^d$ given by
\begin{equation}\label{model}
  \lo=\sum_{j=1}^d\left(Ax\right)_j\frac{\pt}{\pt x_j}+
  \frac{1}{2}\sum_{j,k=1}^d
  \left(BB\dual\right)_{jk}\frac{\pt^2}{\pt x_j\pt x_k}\;.
\end{equation}
For the linear vector field $X_0$ and the constant vector fields
$X_1,\dots,X_m$ on $\R^d$ defined by
\begin{equation*}
  X_0=\sum_{j,k=1}^dA_{jk}x_k\frac{\pt}{\pt x_j}
  \quad\mbox{and}\quad
  X_i=\sum_{j=1}^dB_{ji}\frac{\pt}{\pt x_j}
  \quad\mbox{for } i\in\{1,\dots,m\}\;,
\end{equation*}
the operator $\lo$ rewrites as
\begin{equation*}
  \lo=X_0+\frac{1}{2}\sum_{i=1}^m X_i^2\;.
\end{equation*}
We further note that, for $i\in\{1,\dots,m\}$ and $k\in\N$,
\begin{equation}\label{brackets}
  \left(\operatorname{ad}_{X_0}\right)^k(X_i)=
  \sum_{j=1}^d(-1)^k\left(A^kB\right)_{ji}\frac{\pt}{\pt x_j}\;,
\end{equation}
where $\operatorname{ad}_{X_0}(Y)=[X_0,Y]$. Hence, putting
condition~(\ref{kalman}) on the matrices $A$ and $B$
ensures that any operator of the form~(\ref{model})
satisfies a weak H\"ormander condition. In control theory,
condition~(\ref{kalman}) is also known as the Kalman rank condition,
cf.~\cite[Section~2.3]{kalman_book}. As remarked in~\cite{evolution},
it is indeed of interest to study the operators of the
form~(\ref{model}) and its associated hypoelliptic diffusions
because they arise when linearising the
Fokker-Planck equation. Moreover, this model class contains some strongly
degenerate operators, see Section~\ref{ultra}.

In the analysis of the small-time fluctuations for the
corresponding
hypoelliptic diffusion bridges, it is of advantage that a diffusion
process with generator of the form~(\ref{model}) is always Gaussian
and in particular, that its bridge processes can be written down
explicitly. Additionally, unlike~\cite{BMN}, we do not come across any cut
locus phenomena for this class of diffusions.
Fix $x\in\R^d$ and let $\eps>0$. There exists a diffusion
process $(x_t^\eps)_{t\in[0,1]}$ starting
from $x$ and having generator $\eps\lo$. 
For $y\in\R^d$, let $(z_t^\eps(y))_{t\in[0,1]}$ be the
process obtained by conditioning $(x_t^\eps)_{t\in[0,1]}$
on $x_1^\eps=y$. An explicit expression for the bridge process
$(z_t^\eps(y))_{t\in[0,1]}$ is given in Lemma~\ref{modelbridge}. We
consider these diffusion bridges in the limit $\eps\to 0$.

Using the notion of the matrix exponential of a square matrix, we
set, for $t\in[0,1]$,
\begin{equation}\label{Gamma}
  \Gamma_t^\eps=\int_0^t\e^{-\eps sA}BB\dual\e^{-\eps sA\dual}\dd s\;.
\end{equation}
According to~\cite[Proposition~A.1]{evolution}, the Kalman rank
condition~(\ref{kalman}) implies that the square matrix
$\Gamma_t^\eps$ is invertible for all $t\in(0,1]$.
Let $(\phi_t^\eps(y))_{t\in[0,1]}$ be the 
deterministic path in $\R^d$ defined by
\begin{equation}\label{phi}
  \phi_t^\eps(y)=\e^{\eps tA}x+
  \e^{\eps tA}\Gamma_t^\eps\left(\Gamma_1^\eps\right)^{-1}
  \left(\e^{-\eps A}y-x\right)\;.
\end{equation}
We see that this path describes the leading order behaviour
of the diffusion bridge $(z_t^\eps(y))_{t\in[0,1]}$ as $\eps\to 0$.
Set
$\Omega^{0,0}=\{\omega\in C([0,1],\R^d)\colon\omega_0=0,\omega_1=0\}$.
\begin{thm}\label{LLN}
  For all $x,y\in\R^d$, the processes
  $(z_t^\eps(y)-\phi_t^\eps(y))_{t\in[0,1]}$ converge weakly
  as $\eps\to 0$ to the zero process on the set of continuous loops
  $\Omega^{0,0}$.
\end{thm}
In our discussion of examples in Section~\ref{sec:examples}, we observe that
the path $(\phi_t^\eps(y))_{t\in[0,1]}$ can exhibit a blow-up behaviour
in the limit $\eps\to 0$.
Hence, this path compensates for any blow-up
occurring in the process $(z_t^\eps(y))_{t\in[0,1]}$.
In comparison to the law of large number type
theorem~\cite[Theorem~1.1]{BMN} for sub-Riemannian diffusion bridges,
we note that in the weak H\"ormander setting the minimal-like path
$(\phi_t^\eps(y))_{t\in[0,1]}$ depends on $\eps>0$. However, as
in~\cite[Section~2]{BMN}, the path $(\phi_t^\eps(y))_{t\in[0,1]}$ can
still be obtained as projection of a solution to an appropriate
Hamiltonian system. Let us consider the Hamiltonian
$\ham^\eps\colon T\dual\R^d\to\R$ given by
\begin{equation*}
  \ham^\eps(q,p)=\eps p\dual A q+\frac{1}{2}p\dual B B\dual p\;.
\end{equation*}
The description
in~\cite[Section~2]{barilariandpaoli} implies that
$(\phi_t^\eps(y))_{t\in[0,1]}$ is the projection onto $\R^d$ of the unique
solution in $T\dual\R^d$ to the Hamiltonian equations associated with
$\ham^\eps$ subject to starting in $T_x\dual\R^d$ at
time $0$ and ending in $T_y\dual\R^d$ at time $1$.

Theorem~\ref{LLN} is a consequence of our study of the small-time
fluctuations for the bridge
$(z_t^\eps(y))_{t\in[0,1]}$. To state our fluctuation
result, cf.~Theorem~\ref{CLT}, we first introduce a basis for $\R^d$
which simplifies the analysis, also
see~\cite{barilariandpaoli} and \cite{evolution}.
For $k\in\{1,\dots,n\}$, set
\begin{equation}\label{ek}
  E_k=\operatorname{span}\left\{
    A^l Bv\colon v\in\R^m,\; 0\leq l\leq k-1\right\}\;,
\end{equation}
that is, $E_k$ is the subspace of $\R^d$ defined by the columns of the
matrices $A^lB$ for $l\in\{0,\dots,k-1\}$.
By condition~(\ref{kalman}) and the minimality of $n$, we
know both that $E_n=\R^d$ and that $E_{n-1}$ is a strict subset of
$\R^d$. Set $d_k=\dim E_k$. Since $\{E_k\}_{1\leq k\leq n}$ is an increasing
filtration of subspaces
of $\R^d$, we can and do choose an orthonormal basis
$\{e_1,\dots,e_d\}$ of $\R^d$ such that $\{e_1,\dots,e_{d_k}\}$ is
a basis of $E_k$.
For $r\in\R$, define
\begin{equation}\label{ureps}
  U^\eps(r)=\e^{\eps rA}B\;.
\end{equation}
As detailed in Lemma~\ref{urexpansion}, in the limit $\eps\to 0$ and
in our chosen basis, $U^\eps(r)$ takes the form
\begin{equation*}
  U^\eps(r)=
  \begin{pmatrix}
    u_1 \\ \eps ru_2 \\ \vdots \\ \eps^{n-1}r^{n-1}u_n
  \end{pmatrix}+
  \begin{pmatrix}
    O\left(\eps\right) \\ O\left(\eps^2\right) \\
    \vdots \\ O\left(\eps^n\right)
  \end{pmatrix}\;,
\end{equation*}
where $u_k$ is a $(d_k-d_{k-1}) \times m$ matrix
with constant entries. Here we use the convention that $d_0=0$.
Let
$D_\eps$ and $J_t$ be the $d\times d$ diagonal matrices whose
$j^{\rm th}$
diagonal element, for $d_{k-1}<j\leq d_k$, equals $\eps^{k-1}$ and
$t^{k-1/2}$, respectively.
The natural rescaled fluctuation process to study is
$(F_t^\eps)_{t\in[0,1]}$ given by
\begin{equation}\label{refluct}
  F_t^\eps=\eps^{-1/2}D_\eps^{-1}\left(z_t^\eps(y)-\phi_t^\eps(y)\right)\;,
\end{equation}
where we show that the fluctuations indeed neither depend on $x\in\R^d$
nor on $y\in\R^d$.
As in~\cite{MAP} and due to~(\ref{brackets}),
the orders of $\eps$ which we rescale the fluctuations by are
determined in terms of a filtration
induced by the commutator brackets of the
vector fields $X_0,X_1,\dots,X_m$. 
To describe the limit fluctuation process, we set, for $r\in\R$,
\begin{equation}\label{Uhat}
  \hat U(r)=
  \begin{pmatrix}
    u_1 \\ ru_2 \\ \vdots \\ r^{n-1}u_n
  \end{pmatrix}\;,
\end{equation}
and further introduce the $d\times d$
matrix $V$ which is an $n\times n$ block matrix whose
$(k,l)^{\rm th}$ block~element $V_{kl}$ is the
$(d_k-d_{k-1})\times (d_l-d_{l-1})$ matrix given by
\begin{equation}\label{Vdefn}
  V_{kl}=(-1)^{l+1}u_ku_l\dual\frac{(k-1)!\,(l-1)!}{(k+l-1)!}\;.
\end{equation}
As established in Lemma~\ref{invertV}, the matrix $V$ is
invertible. This allows us to describe the
small-time fluctuations for the bridge process
$(z_t^\eps(y))_{t\in[0,1]}$ as follows.
\begin{thm}\label{CLT}
  Let $(W_t)_{t\in[0,1]}$ be a standard Brownian motion in $\R^m$.
  In the chosen basis of $\R^d$, let $(F_t)_{t\in[0,1]}$ be the
  process defined by
  \begin{equation*}
    F_t=\int_0^t\hat U(t-s)\dd W_s-J_tVJ_tV^{-1}\int_0^1\hat U(1-s)\dd W_s\;.
  \end{equation*}
  Then, for all $x,y\in\R^d$, the rescaled fluctuations
  $(F_t^\eps)_{t\in[0,1]}$ converge weakly
  to $(F_t)_{t\in[0,1]}$ as $\eps\to 0$.
\end{thm}
It is of interest by itself
that after compensating for a blow-up
in
the process $(z_t^\eps(y))_{t\in[0,1]}$
through the path $(\phi_t^\eps(y))_{t\in[0,1]}$,
the small-time fluctuations do not
exhibit any further blow-ups as $\eps\to 0$.
Moreover, the example discussed in
Section~\ref{sec:OU} demonstrates that, while the
bridge processes and the
rescaled fluctuations
can always be computed explicitly due to the Gaussian nature of the
considered diffusion, Theorem~\ref{CLT} indeed simplifies the
determination of the small-time fluctuations for the bridge.

We observe that since $D_\eps,J_t, \hat U(r)$ and $V$ are uniquely
determined in terms of $n\in\N$ and $u_1,\dots,u_n$,
processes which give rise to the same
$n\in\N$ and $u_1,\dots,u_n$ for the same
orthonormal basis of $\R^d$
exhibit the same small-time fluctuations for the bridge, according to
Theorem~\ref{CLT}.
A formulation of this property in terms of the generator
$\lo$ is given in Remark~\ref{nilpotent}.
It is similar to~\cite{MAP} where, in a suitable
chart, the small-time
fluctuations for sub-Riemannian diffusion loops only depend on the nilpotent
approximations of the vector fields $X_1,\dots,X_m$.

The paper is organised as follows.
In Section~\ref{sec:modelbridge}, we discuss in more detail the hypoelliptic
diffusions in our model class, 
and we derive an
expression for the associated bridge processes. The small-time
analysis, which leads to the proofs of Theorem~\ref{LLN} and
Theorem~\ref{CLT}, is then performed in
Section~\ref{sec:smalltime}. We close by presenting a collection of
examples in Section~\ref{sec:examples} to illustrate our results.
As part of the discussions in Section~\ref{ultra}, we find an
explicit expression for the bridge from $0$ to $0$ in time $1$ of an
iterated Kolmogorov diffusion.

%% file: bridge.tex
We analyse the diffusion processes whose generators are of the
form~(\ref{model}) for matrices $A$ and $B$ satisfying 
condition~(\ref{kalman}). We further derive explicit expressions for
the associated bridge processes.
Let $(W_t)_{t\in[0,1]}$ be a standard
Brownian motion in $\R^m$, which we assume is realised as the
coordinate process on the path space
$\{w\in C([0,1],\R^m)\colon w_0=0\}$ under Wiener measure $\pr$.
Fix $x\in\R^d$. For $\eps>0$, let $(x_t^\eps)_{t\in[0,1]}$
be the unique strong solution to the It\^o stochastic differential
equation in $\R^d$
\begin{equation*}
  \db x_t^\eps = \eps A x_t^\eps\dd t+\sqrt{\eps}B\dd W_t\;,
  \qquad x_0^\eps=x\;.
\end{equation*}
We note that the process $(x_t^\eps)_{t\in[0,1]}$ has
generator $\eps\lo$, where
$\lo$ is given by~(\ref{model}). From the
discussions in the
Introduction, we know that operators of this form
satisfy a weak H\"ormander condition and hence, that
$(x_t^\eps)_{t\in[0,1]}$ is a hypoelliptic diffusion.
It has the explicit expression
\begin{equation}\label{expexp}
  x_t^\eps=\e^{\eps tA}x+\e^{\eps tA}\int_0^t\e^{-\eps sA}\sqrt{\eps}B\dd W_s\;,
\end{equation}
as can be checked by direct computation. We see
that $(x_t^\eps)_{t\in[0,1]}$ is a Gaussian process with
\begin{equation}\label{modelmean}
  \E\left[x_t^\eps\right]=\e^{\eps tA}x
\end{equation}
and whose covariance structure is given as follows in terms of
$\Gamma_t^\eps$ defined by~(\ref{Gamma}).
\begin{lemma}\label{modelcov}
  For $t_1,t_2\in[0,1]$ with $t_1\leq t_2$, we have
  \begin{equation*}
    \operatorname{cov}\left(x_{t_1}^\eps,x_{t_2}^\eps\right)
    =\eps\e^{\eps t_1 A}\Gamma_{t_1}^\eps\e^{\eps t_2 A\dual}\;.
  \end{equation*}
\end{lemma}
\begin{proof}
  Using the expression~(\ref{expexp}),
  the property~(\ref{modelmean}) and the It\^o isometry,
  we obtain
  \begin{align*}
    \operatorname{cov}\left(x_{t_1}^\eps,x_{t_2}^\eps\right)
    &=\E\left[\left(x_{t_1}^\eps-\E\left[x_{t_1}^\eps\right]\right)
      \left(x_{t_2}^\eps-\E\left[x_{t_2}^\eps\right]\right)\dual\right]\\
    &=\E\left[\e^{\eps t_1 A}\int_0^{t_1}\e^{-\eps sA}\sqrt{\eps}B\dd W_s
      \left(\int_0^{t_2}\e^{-\eps sA}\sqrt{\eps}B\dd W_s\right)\dual
      \e^{\eps t_2 A\dual}\right]\\
    &=\eps\e^{\eps t_1 A}
      \left(\int_0^{t_1}\e^{-\eps sA}BB\dual\e^{-\eps sA\dual}\dd s\right)
      \e^{\eps t_2 A\dual}\\
    &=\eps\e^{\eps t_1 A}\Gamma_{t_1}^\eps\e^{\eps t_2 A\dual}\;,
  \end{align*}
  as claimed.
\end{proof}
With the covariance structure for the Gaussian process
$(x_t^\eps)_{t\in[0,1]}$ at hand,
we can find an explicit expression
for the corresponding bridge processes. The derivation relies on the fact
that Gaussian random variables are independent if and only if they are
uncorrelated.
\begin{lemma}\label{modelbridge}
  For $t\in[0,1]$, set
  \begin{equation}\label{alpha}
    \alpha_t^\eps=\e^{\eps tA}\Gamma_t^\eps
    \left(\Gamma_1^\eps\right)^{-1}\e^{-\eps A}\;.
  \end{equation}
  Then, for $y\in\R^d$, the stochastic process
  $(z_t^\eps(y))_{t\in[0,1]}$ in $\R^d$ given by
  \begin{equation*}
    z_t^\eps(y)=x_t^\eps-\alpha_t^\eps\left(x_1^\eps-y\right)
  \end{equation*}
  has the same law as the process
  $(x_t^\eps)_{t\in[0,1]}$ conditioned on $x_1^\eps=y$.
\end{lemma}
\begin{proof}
  For all $t\in[0,1]$, we can write
  \begin{equation}\label{split}
    x_t^\eps=z_t^\eps(0)+\alpha_t^\eps x_1^\eps\;.
  \end{equation}
  Applying Lemma~\ref{modelcov}, we compute that
  \begin{align*}
    \operatorname{cov}\left(z_t^\eps(0),x_1^\eps\right)
    &=\operatorname{cov}\left(x_t^\eps-\alpha_t^\eps x_1^\eps,x_1^\eps\right)
    =\operatorname{cov}\left(x_t^\eps,x_1^\eps\right)-
     \alpha_t^\eps\operatorname{cov}\left(x_1^\eps,x_1^\eps\right)\\
    &=\eps\e^{\eps t A}\Gamma_{t}^\eps\e^{\eps A\dual}{}-
      \eps\alpha_t^\eps\e^{\eps A}\Gamma_{1}^\eps\e^{\eps A\dual}=0\;.
  \end{align*}
  Since $(z_t^\eps(0))_{t\in[0,1]}$ and
  $x_1^\eps$ are both Gaussian,
  their vanishing covariance implies that $(z_t^\eps(0))_{t\in[0,1]}$
  and $x_1^\eps$ are independent. Thus, from the
  representation~(\ref{split}) it follows that the bridge obtained by
  conditioning the process $(x_t^\eps)_{t\in[0,1]}$ on $x_1^\eps=y$
  can be expressed, at time $t\in[0,1]$, as
  \begin{equation*}
    z_t^\eps(0)+\alpha_t^\eps y\;,
  \end{equation*}
  which equals $z_t^\eps(y)$.
\end{proof}
We observe that the path $(\phi_t^\eps(y))_{t\in[0,1]}$ in $\R^d$ defined
by~(\ref{phi}) rewrites as
\begin{equation*}
  \phi_t^\eps(y)=\e^{\eps tA}x+\alpha_t^\eps\left(y-\e^{\eps A}x\right)\;.
\end{equation*}
Hence, the expression~(\ref{expexp}) and
Lemma~\ref{modelbridge} imply
\begin{equation}\label{nonscaledfluct}
  z_t^\eps(y)-\phi_t^\eps(y)=
  \e^{\eps tA}\int_0^t\e^{-\eps sA}\sqrt{\eps}B\dd W_s-
  \alpha_t^\eps\left(\e^{\eps A}\int_0^1\e^{-\eps sA}\sqrt{\eps}B\dd W_s\right)\;.
\end{equation}
The analysis of this expression in the limit $\eps\to 0$ is performed
in the next section.

%% file: smalltime.tex
We study the dependence of $\e^{\eps rA}B$ and $\alpha_t^\eps$ given
by~(\ref{alpha}) on $\eps\to 0$ and then use
the expression~(\ref{nonscaledfluct}) to give
the proofs of Theorem~\ref{LLN} and Theorem~\ref{CLT}.
Recall from~(\ref{ureps}) that, for $r\in\R$, we define
\begin{equation*}
  U^\eps(r)=\e^{\eps rA}B\;.
\end{equation*}
In a suitable basis, $U^\eps(r)$ takes the following form.
\begin{lemma}\label{urexpansion}
  Let $\{e_1,\dots,e_d\}$ be an orthonormal basis of $\R^d$ such that
  $\{e_1,\dots,e_{d_k}\}$
  is a basis of the subspace $E_k$ given by~{\rm (\ref{ek})},
  for $k\in\{1,\dots,n\}$. In such a basis, $U^\eps(r)$
  has the form, as $\eps\to 0$,
  \begin{equation}\label{urepsform}
    U^\eps(r)=
    \begin{pmatrix}
      u_1 \\ \eps ru_2 \\ \vdots \\ \eps^{n-1}r^{n-1}u_n
    \end{pmatrix}+
    \begin{pmatrix}
      O\left(\eps\right) \\ O\left(\eps^2\right) \\
      \vdots \\ O\left(\eps^n\right)
    \end{pmatrix}\;,
  \end{equation}
  uniformly in $r$ on compact intervals,
  where $u_k$ is a $(d_k-d_{k-1}) \times m$ matrix with constant
  entries.
\end{lemma}
\begin{proof}
  Write $\langle\cdot,\cdot\rangle$ for the
  standard inner product on $\R^d$.
  Since
  $E_k$ is the subspace of $\R^d$ spanned by the columns of
  $A^lB$ for $l\in\{0,\dots,k-1\}$, these columns
  can be written as a linear
  combination of the vectors $e_1,\dots,e_{d_k}$.
  It follows that, for 
  $j\in\{d_{k}+1,\dots,d\}$ and for all $v\in\R^m$,
  \begin{equation}\label{cancelAlB}
    \left\langle e_j,A^lBv\right\rangle=0
    \quad\mbox{for }l\in\{0,\dots,k-1\}\;.
  \end{equation}
  Due to the properties of the matrix exponential, we have, as $\eps\to 0$,
  \begin{equation*}
    \e^{\eps r A}=\sum_{l=0}^{k-1}\frac{1}{l!}\left(\eps rA\right)^l+
    \frac{1}{k!}\left(\eps rA\right)^k+O\left(\eps^{k+1}\right)\;,
  \end{equation*}
  uniformly in $r\in\R$ on compact intervals.
  By using~(\ref{cancelAlB}) we obtain that,
  for all $j\in\{d_{k}+1,\dots,d\}$ and all $v\in\R^m$,
  \begin{equation}\label{origexp}
    \begin{aligned}
      \left\langle e_j,U^\eps(r)v\right\rangle=
      \left\langle e_j,\e^{\eps r A}Bv\right\rangle
      &=
      \sum_{l=0}^{k-1}\frac{\eps^l r^l}{l!}
      \left\langle e_j,A^lBv\right\rangle+
      \frac{\eps^kr^k}{k!}
      \left\langle e_j,A^kBv\right\rangle+O\left(\eps^{k+1}\right)\\
      &=\frac{\eps^kr^k}{k!}
      \left\langle e_j,A^kBv\right\rangle+O\left(\eps^{k+1}\right)\;,
    \end{aligned}
  \end{equation}
  uniformly in $r$ on compact intervals.
  This establishes that $U^\eps(r)$ is indeed of the
  form~(\ref{urepsform}).
\end{proof}
We work in such an orthonormal basis of $\R^d$ which respects the
filtration of subspaces $\{E_k\}_{1\leq k\leq n}$ for the remainder of
the section. According to Lemma~\ref{urexpansion},
for the rescaling matrix $D_\eps$ and for
$\hat U(r)$ defined by~(\ref{Uhat}), we have
\begin{equation}\label{exp1}
  U^\eps(r)=D_\eps\left(\hat U(r)+O\left(\eps\right)\right)\;,
\end{equation}
uniformly in $r\in\R$ on compact intervals. We deduce that,
uniformly in $t\in[0,1]$,
\begin{align}
  \e^{\eps tA}\Gamma_t^\eps=\label{conciseGamma}
  \e^{\eps tA}\int_0^t\e^{-\eps sA}BB\dual\e^{-\eps sA\dual}\dd s
  &=\int_0^t U^\eps(t-s)U^\eps(-s)\dual\dd s\\
  &=D_\eps\left(\int_0^t \hat U(t-s)\hat U(-s)\dual\dd s+
    O\left(\eps\right)\right)D_\eps\;.\nonumber
\end{align}
We use the following lemma to obtain a concise expression
of $\int_0^t \hat U(t-s)\hat U(-s)\dual\dd s$, for $t\in[0,1]$,
in terms of $u_1,\dots,u_n$.
\begin{lemma}\label{intidentity}
  For $k,l\in\N$ and for all $t\in[0,1]$, we have
  \begin{equation*}
    \int_0^t(t-s)^{k-1}(-s)^{l-1}\dd s=
    (-1)^{l-1}\frac{(k-1)!\,(l-1)!}{(k+l-1)!}t^{k+l-1}\;.
  \end{equation*}
\end{lemma}
\begin{proof}
  We prove this identity by induction over $k\in\N$ with
  $l\in\N$ fixed. For $k=1$, we compute
  \begin{equation*}
    \int_0^t(-s)^{l-1}\dd s=\frac{(-1)^{l-1}t^l}{l}=
    \frac{(-1)^{l-1}(l-1)!}{l!}t^l\;,
  \end{equation*}
  which settles the base case for all $t\in[0,1]$.
  To establish the induction step, consider
  the functions 
  $f_{k},g_{k}\colon[0,1]\to\R$ defined by
  \begin{equation*}
    f_{k}(t)=\int_0^t(t-s)^{k-1}(-s)^{l-1}\dd s
    \quad\mbox{and}\quad
    g_{k}(t)=(-1)^{l-1}\frac{(k-1)!\,(l-1)!}{(k+l-1)!}t^{k+l-1}\;.
  \end{equation*}
  We have
  \begin{equation*}
    \frac{\db}{\db t}f_{k}(t)=
    (k-1)\int_0^t(t-s)^{k-2}(-s)^{l-1}\dd s
  \end{equation*}
  as well as
  \begin{equation*}
    \frac{\db}{\db t}g_{k}(t)=
    (-1)^{l-1}\frac{(k-1)!\,(l-1)!}{(k+l-2)!}t^{k+l-2}\;.
  \end{equation*}
  The induction hypothesis implies that
  \begin{equation}\label{helpeq}
    \frac{\db}{\db t}f_{k}(t)=(k-1)f_{k-1}=(k-1)g_{k-1}=
    \frac{\db}{\db t}g_{k}(t)
  \end{equation}
  for all $t\in[0,1]$. Due to $f_{k}(0)=0=g_{k}(0)$,
  the result follows upon integrating~(\ref{helpeq}).
\end{proof}
For $t\in[0,1]$, the matrix
$\int_0^t\hat U(t-s)\hat U(-s)\dual\dd s$ is an $n\times n$ block matrix
whose $(k,l)^{\rm th}$ block element
is the $(d_k-d_{k-1})\times (d_l-d_{l-1})$ matrix
\begin{equation*}
  u_ku_l\dual\int_0^t(t-s)^{k-1}(-s)^{l-1}\dd s\;.
\end{equation*}
Using Lemma~\ref{intidentity} we deduce
that, with the $n\times n$ block matrix $V$
defined by~(\ref{Vdefn}) and the rescaling matrix $J_t$,
\begin{equation}\label{uhatjvj}
  \int_0^t\hat U(t-s)\hat U(-s)\dual\dd s=J_tVJ_t\;.
\end{equation}
Following on from~(\ref{conciseGamma}),
we end up with the expression
\begin{equation}\label{expGammaexp}
  \e^{\eps tA}\Gamma_t^\eps=
  D_\eps\left(\int_0^t \hat U(t-s)\hat U(-s)\dual\dd s+
    O\left(\eps\right)\right)D_\eps=
  D_\eps J_t\left(V+O\left(\eps\right)\right)J_tD_\eps\;,
\end{equation}
uniformly in $t\in[0,1]$.
To use~(\ref{expGammaexp}) to obtain an alternative expression
for $\alpha_t^\eps$, we first show that the square matrix $V$ is
invertible.
\begin{lemma}\label{invertV}
  The $n\times n$ block matrix $V$ whose $(k,l)^{\rm th}$ block
  element is given by~{\rm (\ref{Vdefn})} is invertible.
\end{lemma}
\begin{proof}
  As shown in~\cite[Proposition~2.1]{evolution}, in our chosen basis
  of $\R^d$, the matrix $A$ takes the form of an
  $n\times n$ block matrix whose $(k,l)^{\rm th}$ block element, for
  $k,l\in\{1,\dots,n\}$, is a $(d_k-d_{k-1})\times (d_l-d_{l-1})$
  matrix, where all the blocks with $k\geq l+2$ vanish.
  Let $\hat A$ be an $n\times n$ block matrix of the same block
  structure. We set its block elements to zero unless $k=l+1$, in
  which case we set that block element to equal
  the $(k,l)^{\rm th}$ block element of $A$. By definition of the
  subspace $E_1$ of $\R^d$, we further observe that in our chosen basis,
  for all $j\in\{d_1+1,\dots, d\}$, the $j^{\rm th}$ row of $B$
  vanishes.
  
  For $l\in\{1,\dots,n-1\}$, let $A_l$ denote the
  $(l+1,l)^{\rm th}$ block element of the matrix $A$ and let
  $B_1$ be the $d_1\times m$ matrix obtained by considering the first
  $d_1$ rows of $B$ only. For $k\in\{1,\dots,n\}$, we set
  \begin{equation*}
    \hat E_k=\operatorname{span}
    \left\{\hat A^lBv\colon v\in\R^m,\; 0\leq l\leq k-1\right\}\;.
  \end{equation*}
  By construction of $\hat A$, the $d\times m$ matrix $\hat A^l B$
  is an $n\times 1$ block matrix, whose
  $(k,1)^{\rm th}$ block element
  is a $(d_k-d_{k-1})\times m$ matrix, which
  vanishes unless $k=l+1$, in which case it equals
  $A_l\cdots A_1 B_1$.
  From this form it follows that, in the chosen basis
  $\{e_1,\dots,e_d\}$ of $\R^d$, we have, for all
  $l\in\{0,\dots,n-1\}$ and all $v\in\R^m$,
  \begin{equation}\label{neweq1}
    \left\langle e_j, \hat A^lBv \right\rangle=0
    \quad\mbox{unless }
    j\in\{d_l+1,\dots,d_{l+1}\}\;.
  \end{equation}
  Moreover, for $l\geq n$, we obtain $\hat A^lB=0$, which implies that,
  for $r\in\R$,
  \begin{equation}\label{neweq2}
    \e^{\eps r \hat A}B=
    \sum_{l=0}^{n-1}\frac{\eps^l r^l}{l!}\hat A^lB\;.
  \end{equation}
  Combining (\ref{neweq1}) and (\ref{neweq2}) yields, for all $v\in\R^m$,
  \begin{equation}\label{hatexp}
    \left\langle e_j,\e^{\eps r \hat A}B v\right\rangle=
    \frac{\eps^lr^l}{l!}\left\langle e_j,\hat A^lBv\right\rangle
    \quad\mbox{for } j\mbox{ with } d_l< j\leq d_{l+1}\;.
  \end{equation}
  After understanding $A^lB$ as an $n\times 1$
  block matrix of the same structure as the matrix $\hat A^l B$,
  we further see that the $(l+1,1)^{\rm th}$ block element of
  $A^lB$ also equals $A_l\cdots A_1B_1$. This is a
  consequence of the observation that a block element in $A$ with
  $k\geq l+2$ vanishes. In particular, for $v\in\R^m$ and
  $j$ with $d_l< j\leq d_{l+1}$, we have
  \begin{equation*}
    \left\langle e_j,\hat A^lBv\right\rangle=
    \left\langle e_j,A^lBv\right\rangle\;,
  \end{equation*}
  and~(\ref{origexp}) together with~(\ref{hatexp}) implies that, for
  $r\in\R$,
  \begin{equation*}
    \hat U(r)=\e^{r \hat A}B\;.
  \end{equation*}
  Using~(\ref{uhatjvj}), we conclude that
  \begin{equation*}
    V=\int_0^1\hat U(1-s)\hat U(-s)\dual\dd s=
    \e^{\hat A}\int_0^1\e^{-s \hat A}BB\dual\e^{-s \hat A\dual}\dd s\;.
  \end{equation*}
  Our discussion above shows that $E_k=\hat E_k$, for all
  $k\in\{1,\dots,n\}$, and especially $\hat E_n=\R^d$.
  Therefore, the matrices
  $\hat A$ and $B$ satisfy the Kalman rank condition, which
  ensures that
  \begin{equation*}
    \int_0^1\e^{-s \hat A}BB\dual\e^{-s \hat A\dual}\dd s
  \end{equation*}
  is invertible. Since $\e^{\hat A}$ has the matrix inverse
  $\e^{-\hat A}$, the invertibility of $V$ follows.
\end{proof}
For completeness, we note that Lemma~\ref{invertV} implies that, for
all $k\in\{1,\dots,n\}$, the matrix $u_k$ has maximal rank. If it did
not then, since $d_k-d_{k-1}\leq m$ by construction,
its rows would be linearly dependent leading to $V$
having a collection of linearly dependent rows, which is not possible.
\begin{remark}\label{nilpotent}
  Let $\hat A$ be the $d\times d$ matrix constructed from the matrix
  $A$ as in the previous proof, and let $\hat\lo$ be the operator on $\R^d$
  given by
  \begin{equation*}
    \hat \lo=\sum_{j=1}^d\left(\hat Ax\right)_j\frac{\pt}{\pt x_j}+
    \frac{1}{2}\sum_{j,k=1}^d
    \left(BB\dual\right)_{jk}\frac{\pt^2}{\pt x_j\pt x_k}\;.
  \end{equation*}
  In~\cite{evolution}, the operator $\hat\lo-\frac{\pt}{\pt t}$ is
  called the principal part of $\lo-\frac{\pt}{\pt t}$, and it is
  shown that the fundamental solution with pole at zero 
  of $\lo-\frac{\pt}{\pt t}$ can
  be controlled in terms of the fundamental solution
  with pole at zero of $\hat\lo-\frac{\pt}{\pt t}$,
  cf.~\cite[Theorem~3.1]{evolution}.
  Similarly, let us call $\hat\lo$ the principal part of $\lo$.

  In our model class of hypoelliptic diffusions
  the small-time fluctuations
  for the bridge are given by Theorem~\ref{CLT} in terms of 
  $D_\eps,J_t,\hat U(r)$ and $V$,
  which due to the proof of Lemma~\ref{invertV} can be
  uniquely determined from $\hat A$ and $B$. Therefore, the
  small-time fluctuations for the bridge are fully governed by the
  principal part $\hat\lo$ of the generator $\lo$.
  A similar property was observed in~\cite{MAP}.
\end{remark}
We now proceed with our analysis to find an alternative expression
for $\alpha_t^\eps$.
Since the set of invertible matrices is open, Lemma~\ref{invertV}
shows that, for $\eps>0$ sufficiently small, the inverse of
$V+O\left(\eps\right)$ exists. It satisfies
\begin{equation*}
  \left(V+O\left(\eps\right)\right)^{-1}=
  V^{-1}+O\left(\eps\right)\;.
\end{equation*}
From~(\ref{expGammaexp}) and as $J_1$ equals the identity matrix,
it follows that
\begin{equation*}
  \left(\Gamma_1^\eps\right)^{-1}\e^{-\eps A}=
  D_\eps^{-1} \left(V^{-1}+O\left(\eps\right)\right)D_\eps^{-1}\;,
\end{equation*}
which yields
\begin{equation}\label{exp2}
  \alpha_t^\eps=\e^{\eps tA}\Gamma_t^\eps
  \left(\Gamma_1^\eps\right)^{-1}\e^{-\eps A}=
  D_\eps\left(J_t V J_t V^{-1}+O\left(\eps\right)\right)D_\eps^{-1}\;,
\end{equation}
uniformly in $t\in[0,1]$.
The two estimates~(\ref{exp1}) and (\ref{exp2})
are the essential ingredients for proving
Theorem~\ref{LLN} and Theorem~\ref{CLT}.
\begin{proof}[Proof of Theorem~\ref{CLT}]
  Using~(\ref{ureps}),
  we can rewrite~(\ref{nonscaledfluct}) as
  \begin{equation}\label{rewrite24}
    z_t^\eps(y)-\phi_t^\eps(y)=
    \sqrt{\eps}\left(\int_0^tU^\eps(t-s)\dd W_s-
      \alpha_t^\eps\int_0^1U^\eps(1-s)\dd W_s\right)\;,
  \end{equation}
  and therefore,
  \begin{equation*}
    F_t^\eps=\eps^{-1/2}D_\eps^{-1}\left(z_t^\eps(y)-\phi_t^\eps(y)\right)=
    \int_0^t D_\eps^{-1}U^\eps(t-s)\dd W_s-
    D_\eps^{-1}\alpha_t^\eps D_\eps\int_0^1D_\eps^{-1}U^\eps(1-s)\dd W_s\;.
  \end{equation*}
  The estimate~(\ref{exp1}) gives
  \begin{equation*}
    \sup_{r\in[0,1]}
    \left\|D_\eps^{-1}U^\eps(r)-\hat U(r)\right\|\to 0
    \quad\mbox{as}\quad\eps\to 0\;,
  \end{equation*}
  whereas~(\ref{exp2}) implies that
  \begin{equation*}
    \sup_{t\in[0,1]}
    \left\|D_\eps^{-1}\alpha_t^\eps D_\eps-J_t V J_t V^{-1}\right\|\to 0
    \quad\mbox{as}\quad\eps\to 0\;.
  \end{equation*}
  Hence, the covariances of the
  mean-zero Gaussian processes $(F_t^\eps)_{t\in[0,1]}$
  converge uniformly as $\eps\to 0$ to the covariance of the
  mean-zero Gaussian process $(F_t)_{t\in[0,1]}$
  given by
  \begin{equation*}
    F_t=\int_0^t\hat U(t-s)\dd W_s-J_tVJ_tV^{-1}\int_0^1\hat U(1-s)\dd W_s\;.
  \end{equation*}
  From~\cite[Section~3]{gaussian}, it follows that
  the rescaled
  fluctuations $(F_t^\eps)_{t\in[0,1]}$ indeed converge weakly to
  $(F_t)_{t\in[0,1]}$ as $\eps\to 0$.
\end{proof}
\begin{proof}[Proof of Theorem~\ref{LLN}]
  Since the rescaled fluctuations $(F_t^\eps)_{t\in[0,1]}$ defined by
  \begin{equation*}
    F_t^\eps=\eps^{-1/2}D_\eps^{-1}\left(z_t^\eps(y)-\phi_t^\eps(y)\right)
  \end{equation*}
  converge weakly as $\eps\to 0$ to a well-defined limit process,
  the processes
  $(z_t^\eps(y)-\phi_t^\eps(y))_{t\in[0,1]}$ converge weakly as
  $\eps\to 0$ to the zero process on the set of loops $\Omega^{0,0}$.
\end{proof}
Before we move on to a discussion of four examples in the following
section, we make an 
observation regarding the process
\begin{equation}\label{almostkol}
  \left(\int_0^t\hat U(t-s)\dd W_s\right)_{t\in[0,1]}\;.
\end{equation}
By integration by parts, we have that, for $k\in\N$,
\begin{equation*}
  \int_0^t(t-s)^{k}\dd W_s=k!
  \int_0^t\int_0^{s_{k}}\dots \int_0^{s_2} W_{s_1}\dd s_1\dots \dd s_{k}\;.
\end{equation*}
Thus, the process~(\ref{almostkol})
can be expressed solely in terms of the matrices $u_1,\dots,u_n$ and
an iterated
Kolmogorov diffusion, that is, a standard 
Brownian motion together with a finite number of its iterated time
integrals.
Since the iterated Kolmogorov diffusion arises as a canonical
example, we determine its small-time fluctuations for the bridge in
Section~\ref{ultra}.
The Kolmogorov diffusion
is discussed separately
as a first example in Section~\ref{sec:kol}
as it already exhibits interesting features.

%% file: examples.tex
We discuss four examples which illustrate different aspects of
Theorem~\ref{LLN} and Theorem~\ref{CLT}.

\subsection{Kolmogorov diffusion}
\label{sec:kol}
The Kolmogorov diffusion, named after
Kolmogorov~\cite{kolmogorov}, is the simplest
example of a stochastic process which satisfies a weak H\"ormander
condition but not the strong H\"ormander condition. It is the
diffusion $(x_t)_{t\in[0,1]}$ in $\R^2$ which pairs a standard Brownian
motion $(W_t)_{t\in[0,1]}$ in $\R$ with its time integral, that is,
\begin{equation*}
  x_t=\left(W_t,\int_0^tW_s\dd s\right)\;.
\end{equation*}
It is the unique strong solution to the
stochastic differential equation
\begin{align*}
  \db\left(x_t\right)_1&=\db W_t\;,\\
  \db\left(x_t\right)_2&=\left(x_t\right)_1\dd t\;,
\end{align*}
subject to $x_0=0$. This process falls into our model class of hypoelliptic
diffusions by taking
\begin{equation*}
  A=
  \begin{pmatrix}
    0 & 0\\
    1 & 0
  \end{pmatrix}
  \quad\mbox{and}\quad
  B=
  \begin{pmatrix}
    1 \\ 0
  \end{pmatrix}\;,
\end{equation*}
which corresponds to the operator
\begin{equation*}
  \lo=x_1\frac{\pt}{\pt x_2}+\frac{1}{2}
  \left(\frac{\pt}{\pt x_1}\right)^2
\end{equation*}
on $\R^2$.
The Kalman rank condition~(\ref{kalman}) is satisfied because
\begin{equation*}
  AB=
  \begin{pmatrix}
    0 \\ 1
  \end{pmatrix}
\end{equation*}
implies $E_2=\R^2$. We first use Lemma~\ref{modelbridge} to determine
the expressions for the associated diffusion bridges in small time to
then explicitly see that Theorem~\ref{LLN} and Theorem~\ref{CLT}
hold. For $\eps>0$, the rescaled Kolmogorov diffusion
$(x_t^\eps)_{t\in[0,1]}$ with generator $\eps\lo$ is
given by
\begin{equation*}
  x_t^\eps=\left(\eps^{1/2}W_t,\eps^{3/2}\int_0^tW_s\dd s\right)\;.
\end{equation*}
Since $A^2=0$, we obtain, for $r\in\R$,
\begin{equation*}
  \e^{\eps r A}=I+\eps rA=
  \begin{pmatrix}
    1      & 0 \\
    \eps r & 1
  \end{pmatrix}\;.
\end{equation*}
We further compute, for $t\in[0,1]$,
\begin{equation*}
  \Gamma_t^\eps=\int_0^t\e^{-\eps sA}B B\dual\e^{-\eps s A\dual}\dd s
  =
  \begin{pmatrix}
    t & -\frac{1}{2}\eps t^2\\[0.4em]
    -\frac{1}{2}\eps t^2 & \frac{1}{3}\eps^2 t^3
  \end{pmatrix}\;.
\end{equation*}
It follows that
\begin{equation*}
  \e^{\eps tA}\Gamma_t^\eps=
  \begin{pmatrix}
    t & -\frac{1}{2}\eps t^2\\[0.4em]
    \frac{1}{2}\eps t^2 & -\frac{1}{6}\eps^2 t^3
  \end{pmatrix}
  \quad\mbox{and}\quad
  \left(\Gamma_1^\eps\right)^{-1}\e^{-\eps A}=
  \begin{pmatrix}
    -2 & 6 \eps^{-1}\\[0.4em]
    -6 \eps^{-1} & 12\eps^{-2}
  \end{pmatrix}\;,
\end{equation*}
which implies
\begin{equation*}
  \alpha_t^\eps=\e^{\eps tA}\Gamma_t^\eps\left(\Gamma_1^\eps\right)^{-1}
  \e^{-\eps A}=
  \begin{pmatrix}
    3t^2-2t & \left(6t-6t^2\right)\eps^{-1} \\[0.4em]
    \left(t^3-t^2\right)\eps & 3t^2-2t^3
  \end{pmatrix}\;.
\end{equation*}
Thus, by Lemma~\ref{modelbridge} and for $y=(a,b)\in\R^2$, the process
$(z_t^\eps(a,b))_{t\in[0,1]}$ in $\R^2$ given by
\begin{align*}
  \left(z_t^{\eps}(a,b)\right)_1
  &=\left(3t^2-2t\right)a+\left(6t-6t^2\right)\frac{b}{\eps}+
    \eps^{1/2}\left(W_t-\left(3t^2-2t\right)W_1-
    \left(6t-6t^2\right)\int_0^1W_s\dd s\right)\;,\\
  \left(z_t^{\eps}(a,b)\right)_2
  &=\left(t^3-t^2\right)a\eps+\left(3t^2-2t^3\right)b+
    \eps^{3/2}\left(\int_0^tW_s\dd s-\left(t^3-t^2\right)W_1-
    \left(3t^2-2t^3\right)\int_0^1W_s\dd s\right)
\end{align*}
has the same law as the rescaled Kolmogorov diffusion
$(x_t^\eps)_{t\in[0,1]}$  conditioned on $x_1^\eps=(a,b)$. From the
explicit expression, it follows that the processes
\begin{equation}\label{koldiff}
  \left(\left(z_t^{\eps}(a,b)\right)_1-
    \left(3t^2-2t\right)a-\left(6t-6t^2\right)\frac{b}{\eps},
    \left(z_t^{\eps}(a,b)\right)_2-
  \left(t^3-t^2\right)a\eps-\left(3t^2-2t^3\right)b\right)_{t\in[0,1]}
\end{equation}
converge weakly as $\eps\to 0$ to the zero process. This is consistent
with Theorem~\ref{LLN} because for the Kolmogorov diffusion starting
from $x=0$, we have
\begin{equation*}
  \phi_t^\eps(y)=\alpha_t^\eps y=
  \left(\left(3t^2-2t\right)a+\left(6t-6t^2\right)\frac{b}{\eps},
  \left(t^3-t^2\right)a\eps+\left(3t^2-2t^3\right)b\right)\;.
\end{equation*}
We note that while the path
$\left(\phi_t^\eps(y)\right)_{t\in[0,1]}$ is well-defined for each
$\eps>0$, its first component blows up as $\eps\to 0$, unless $b=0$.
From the above expression for a Kolmogorov bridge from
$0$ to $(a,b)$ in small time, we further see that rescaling the
processes~(\ref{koldiff}) by $\eps^{1/2}$ in the first
component and by $\eps^{3/2}$ in the second component leads to the
fluctuation process which, at $t\in[0,1]$, is given as
\begin{equation*}
  \left(W_t-\left(3t^2-2t\right)W_1-
    \left(6t-6t^2\right)\int_0^1W_s\dd s,
  \int_0^tW_s\dd s-\left(t^3-t^2\right)W_1-
    \left(3t^2-2t^3\right)\int_0^1W_s\dd s\right)\;.
\end{equation*}
Since this coincides with the expression for
$z_t^1(0)$, the resulting limit
fluctuations are equal in law to a Kolmogorov
bridge from $0$ to $0$ in time~$1$.
Below we conclude that this is also what is given to us by
Theorem~\ref{CLT}. For $r\in\R$, we have
\begin{equation*}
  U^\eps(r)= \e^{\eps r A}B=
  \begin{pmatrix}
    1      \\ \eps r
  \end{pmatrix}\;,
\end{equation*}
which is of the form~(\ref{urepsform})
with $u_1=u_2=1$. In particular, we
already work in a suitable basis. The rescaling map
$D_\eps$ is then
\begin{equation}\label{kolD}
  D_\eps=
  \begin{pmatrix}
    1 & 0 \\ 0 & \eps
  \end{pmatrix}\;.
\end{equation}
Since we consider the rescaled fluctuations $(F_t^\eps)_{t\in[0,1]}$
defined by~(\ref{refluct}) this corresponds to rescaling the
first component by $\eps^{1/2}$ and the second component by
$\eps^{3/2}$, as above. We further obtain that
\begin{equation}\label{kolJ}
  J_t=
  \begin{pmatrix}
    t^{1/2} & 0 \\ 0 & t^{3/2}
  \end{pmatrix}\;,\quad  
  \hat U(r)=
  \begin{pmatrix}
    1 \\ r
  \end{pmatrix}
  \quad\mbox{and}\quad
  V=
  \begin{pmatrix}
    1 & -\frac{1}{2} \\[0.4em] \frac{1}{2} & -\frac{1}{6}
  \end{pmatrix}\;. 
\end{equation}
By integration by parts, we have
\begin{equation}\label{kolhatU}
  \int_0^t\hat U(t-s)\dd W_s=
  \left(W_t,\int_0^t(t-s)\dd W_s\right)=
  \left(W_t,\int_0^tW_s\dd s\right)\;.
\end{equation}
This together with the computation
\begin{equation}\label{kolcov}
  J_t V J_t V^{-1}=
  \begin{pmatrix}
    t & -\frac{1}{2} t^2\\[0.4em]
    \frac{1}{2} t^2 & -\frac{1}{6} t^3
  \end{pmatrix}
  \begin{pmatrix}
    -2 & 6 \\[0.4em]
    -6 & 12
  \end{pmatrix}= 
  \begin{pmatrix}
    3t^2-2t & 6t-6t^2 \\[0.4em]
    t^3-t^2 & 3t^2-2t^3
  \end{pmatrix}
\end{equation}
shows that Theorem~\ref{CLT} indeed yields the same small-time
fluctuations for a Kolmogorov bridge as derived above. Irrespective of
the initial and final positions,
the small-time fluctuations
are equal in law to a
Kolmogorov bridge from $0$ to $0$ in time $1$.

\subsection{Ornstein-Uhlenbeck process paired with its area}
\label{sec:OU}
Performing the small-time analysis for the bridge of an
Ornstein-Uhlenbeck process paired with its area demonstrates that
Theorem~\ref{CLT} can greatly simplify the 
determination of the small-time fluctuations for the bridge.
Let $(W_t)_{t\in[0,1]}$ be a standard Brownian motion in $\R$
and fix $x\in\R^2$.
We consider
the diffusion $(x_t)_{t\in[0,1]}$ in $\R^2$
which is the
unique strong solution to the stochastic differential
equation
\begin{align*}
  \db \left(x_t\right)_1&=-\left(x_t\right)_1\dd t +\db W_t\;,\\
  \db \left(x_t\right)_2&=\left(x_t\right)_1\dd t\;,
\end{align*}
subject to the initial condition
$x_0=x$. This corresponds to the choice
\begin{equation*}
  A=
  \begin{pmatrix}
    -1 & 0\\
    1        & 0
  \end{pmatrix}
  \quad\mbox{and}\quad
  B=
  \begin{pmatrix}
    1 \\ 0
  \end{pmatrix}
\end{equation*}
in our model class of diffusion processes. The matrices $A$ and $B$
satisfy condition~(\ref{kalman}) since
\begin{equation*}
  \operatorname{span}\left\{
  \begin{pmatrix}
    1 \\ 0
  \end{pmatrix},
  \begin{pmatrix}
    -1 & 0\\
    1        & 0
  \end{pmatrix}
  \begin{pmatrix}
    1 \\ 0
  \end{pmatrix}  
  \right\}=\R^2\;.
\end{equation*}
In the following, we first use Lemma~\ref{modelbridge} to find
explicit expressions
for the corresponding
bridge processes in small time to then determine the small-time
fluctuations for the bridge by hand, before we show that
Theorem~\ref{CLT} greatly simplifies the analysis.
Using $A^k=(-1)^{k-1}A$ for $k\in\N$, we compute, for $\eps>0$
and $r\in\R$,
\begin{equation*}
  \e^{\eps rA}=
  \begin{pmatrix}
    \e^{-\eps r} & 0\\[0.4em]
    1-\e^{-\eps r} & 1
  \end{pmatrix}\;.
\end{equation*}
It follows that, for $t\in[0,1]$,
\begin{equation*}
  \Gamma_t^\eps=
  \begin{pmatrix}
    \frac{1}{2\eps}\left(\e^{2\eps t}-1\right) &
    -\frac{1}{2\eps}\left(\e^{\eps t}-1\right)^2 \\[0.4em]
    -\frac{1}{2\eps}\left(\e^{\eps t}-1\right)^2 &
    \frac{1}{2\eps}\left(\e^{2\eps t}-4\e^{\eps t}+2\eps t +3\right)
  \end{pmatrix}\;,
\end{equation*}
which yields
\begin{equation*}
  \e^{\eps tA}\Gamma_t^\eps=
  \begin{pmatrix}
    \frac{1}{2\eps}\left(\e^{\eps t}-\e^{-\eps t}\right) &
    -\frac{1}{2\eps}\left(\e^{\eps t}+\e^{-\eps t}-2\right)\\[0.4em]
    \frac{1}{2\eps}\left(\e^{\eps t}+\e^{-\eps t}-2\right) &
    -\frac{1}{2\eps}\left(\e^{\eps t}-\e^{-\eps t}-2\eps t\right)
  \end{pmatrix}\;.
\end{equation*}
A straightforward but
elaborate calculation shows that, for $t\in[0,1]$, the matrix
$\alpha_t^\eps$ is given by
\begin{align*}
  \left(\alpha_t^{\eps}\right)_{11}
  &=\frac{\left(1-\e^{-\eps t}\right)\left((\eps-1)\e^{\eps(1+t)}+\e^{\eps t}+
                       (\eps+1)\e^\eps-\e^{2\eps}\right)}
    {\left(\e^\eps-1\right)\left((\eps-2)\e^\eps+\eps+2\right)}\;,\\
  \left(\alpha_t^{\eps}\right)_{12}
  &=\frac{\e^{-\eps}-\e^{-\eps(1-t)}-\e^{-\eps t}+
    1}{(\eps+2)\e^{-\eps}+\eps-2}\;,\\
  \left(\alpha_t^{\eps}\right)_{21}
  &=\frac{\e^{2\eps}-1+(\eps +1)\e^{\eps(1-t)}+\e^{\eps t}+(\eps-1)\e^{\eps(1+t)}
          -\eps t\left(\e^\eps-1\right)^2-2\eps\e^{\eps}-\e^{\eps(2-t)}}
    {\left(\e^\eps-1\right)\left((\eps-2)\e^\eps+\eps+2\right)}\;,\\
  \left(\alpha_t^{\eps}\right)_{22}
  &=\frac{\e^{-\eps t}-\e^{-\eps(1-t)}+(\eps t +1) \e^{-\eps}
          +\eps t -1}{(\eps+2)\e^{-\eps}+\eps-2}\;.
\end{align*}
By Lemma~\ref{modelbridge}, this gives an explicit expression
for the bridge of the considered Ornstein-Uhlenbeck process
paired with its area. Repeatedly applying l'H\^opital's rule, we
see that, as $\eps\to 0$,
\begin{align*}
  \left(\alpha_t^{\eps}\right)_{11}
  &=3t^2-2t+O\left(\eps^2\right)\;,\\
  \left(\alpha_t^{\eps}\right)_{12}
  &=\frac{6t-6t^2}{\eps}+O\left(\eps\right)\;,\\
  \left(\alpha_t^{\eps}\right)_{21}
  &=\left(t^3-t^2\right)\eps+O\left(\eps^3\right)\;,\\
  \left(\alpha_t^{\eps}\right)_{22}
  &=3t^2-2t^3+O\left(\eps^2\right)\;,
\end{align*}
uniformly in $t\in[0,1]$.
From a comparison to the expressions we obtained in the small-time
analysis for the Kolmogorov diffusion in Section~\ref{sec:kol}, we
deduce that the Ornstein-Uhlenbeck process paired with its
area exhibits the same small-time fluctuations for the bridge as the
Kolmogorov diffusion. This follows much more easily by
applying Theorem~\ref{CLT}. For $r\in\R$, we have
\begin{equation*}
  U^\eps(r)=
  \begin{pmatrix}
    \e^{-\eps r} \\ 1-\e^{-\eps r}
  \end{pmatrix}=
  \begin{pmatrix}
    1 \\ \eps r
  \end{pmatrix}+
  \begin{pmatrix}
    O\left(\eps\right) \\ O\left(\eps^2\right)
  \end{pmatrix}\;.
\end{equation*}
We see that $U^\eps(r)$ is of the form~(\ref{urepsform})
with $u_1=u_2=1$. Hence, the rescaling matrix $D_\eps$ as well as $J_t,
\hat U(r)$ and $V$ are again given by~(\ref{kolD})
as well as~(\ref{kolJ}).
Similarly, the quantities~(\ref{kolhatU}) and (\ref{kolcov}), which
characterise the small-time fluctuations uniquely, remain unchanged.
Thus, as a result of giving rise to the same $\hat U(r)$ for all
$r\in\R$, 
the Kolmogorov diffusion and the Ornstein-Uhlenbeck process paired
with its area exhibit the same small-time
fluctuations for the bridge.

For $x=0$
there is another interesting observation we can make in
regards to these two processes,
which is a consequence of certain terms vanishing in the
Laurent expansion of $\alpha_t^\eps$ in $\eps\to 0$. If we consider
the path $(\psi_t^\eps(y))_{t\in[0,1]}$
defined by, for $y=(a,b)$,
\begin{equation*}
  \psi_t^\eps(y)=\left(
    \left(3t^2-2t\right)a+\left(6t-6t^2\right)\frac{b}{\eps},
    \left(t^3-t^2\right)a\eps+\left(3t^2-2t^3\right)b
  \right)
\end{equation*}
then this is sufficient to compensate for the blow-up behaviour in the
bridge process $(z_t^\eps(y))_{t\in[0,1]}$ as $\eps\to 0$, and the
two processes
\begin{equation*}
  \left(\eps^{-1/2}D_\eps^{-1}
    \left(z_t^\eps(y)-\phi_t^\eps(y)\right)\right)_{t\in[0,1]}
  \quad\mbox{and}\quad
  \left(\eps^{-1/2}D_\eps^{-1}
    \left(z_t^\eps(y)-\psi_t^\eps(y)\right)\right)_{t\in[0,1]}
\end{equation*}
have the same limit process as $\eps\to 0$.
Since the approximate minimal-like path $(\psi_t^\eps(y))_{t\in[0,1]}$
for the current example with $x=0$
coincides with the minimal-like path for a Kolmogorov bridge from $0$
to $y$ in small time, not only the small-time fluctuations
for the bridge but also a sufficiently good approximation of the 
minimal-like path is given in terms of the Kolmogorov diffusion.
Though, as shown in the next example, 
the latter need not hold for two processes which admit the same
$n\in\N$ and $u_1,\dots,u_n$.

%% file: diffmin.tex
\subsection{Compensating for blow-ups in the bridge process}
\label{sec:diffmin}
While the small-time
fluctuations for the bridge are uniquely
determined in terms of the matrices $u_1,\dots,u_n$, we present an
example which shows that knowledge of $u_1,\dots,u_n$ is not
sufficient to construct a path which approximates
the minimal-like path well enough to recover the limit
fluctuations as in the previous section.
We consider the hypoelliptic diffusion corresponding to
the matrices
\begin{equation*}
  A=
  \begin{pmatrix}
    -1 & 0 \\
    1  & 2 \\
  \end{pmatrix}
  \quad\mbox{and}\quad
  B=
  \begin{pmatrix}
    1 \\ 0
  \end{pmatrix}\;,
\end{equation*}  
which satisfy condition~(\ref{kalman}).
Using the eigendecomposition of $A$, we obtain,
for $\eps>0$ and $r\in\R$,
\begin{equation*}
  \e^{\eps rA}=
  \begin{pmatrix}
    \e^{-\eps r} & 0\\[0.4em]
    \frac{1}{3}\left(\e^{2\eps r}-\e^{-\eps r}\right) & \e^{2\eps r}
  \end{pmatrix}
\end{equation*}
as well as
\begin{equation*}
  U^\eps(r)=
  \begin{pmatrix}
    \e^{-\eps r} \\[0.4em]
    \frac{1}{3}\left(\e^{2\eps r}-\e^{-\eps r}\right)
  \end{pmatrix}=
  \begin{pmatrix}
    1 \\[0.4em] \eps r
  \end{pmatrix}+
  \begin{pmatrix}
    O\left(\eps\right) \\[0.4em] O\left(\eps^2\right)
  \end{pmatrix}\;,
\end{equation*}
uniformly in $r$ on compact intervals. Thus, as for the
Ornstein-Uhlenbeck process paired with its area and the
Kolmogorov diffusion, $U^\eps(r)$ is of the form~(\ref{urepsform})
with $u_1=u_2=1$. By Theorem~\ref{CLT},
these three processes exhibit the same
small-time fluctuations for the bridge.
We further compute that, for $t\in[0,1]$,
\begin{equation*}
  \e^{\eps tA}\Gamma_t^\eps=
  \begin{pmatrix}
    \frac{1}{2\eps}\left(\e^{\eps t}-\e^{-\eps t}\right) &
    -\frac{1}{6\eps}\left(2\e^{-2\eps t}-3\e^{-\eps t}+\e^{\eps t}\right)\\[0.5em]
    \frac{1}{6\eps}\left(2\e^{2\eps t}-3\e^{\eps t}+\e^{-\eps t}\right) &
    -\frac{1}{12\eps}
    \left(\e^{2\eps t}-2\e^{\eps t}+2\e^{-\eps t}-\e^{-2\eps t}\right)
  \end{pmatrix}\;,
\end{equation*}
which has the expansion
\begin{equation*}
  \e^{\eps tA}\Gamma_t^\eps=
  \begin{pmatrix}
    t+\frac{1}{6}\eps^2 t^3+O\left(\eps^4\right) &
    -\frac{1}{2}\eps t^2+\frac{1}{3}\eps^2t^3+O\left(\eps^3\right)\\[0.5em]
    \frac{1}{2}\eps t^2+\frac{1}{3}\eps^2t^3+O\left(\eps^3\right) &
    -\frac{1}{6}\eps^2t^3-\frac{1}{24}\eps^4 t^5+O\left(\eps^6\right)
  \end{pmatrix}\;,
\end{equation*}
uniformly in $t\in[0,1]$. Setting
\begin{equation*}
  R=
  \begin{pmatrix}
    0 & \frac{1}{3}\\[0.4em]
    \frac{1}{3} & 0
  \end{pmatrix}
\end{equation*}
and with
\begin{equation*}
  D_\eps=
  \begin{pmatrix}
    1 & 0\\
    0 & \eps
  \end{pmatrix}\;,\quad
  J_t=
  \begin{pmatrix}
    t^{1/2} & 0\\
    0 & t^{3/2}
  \end{pmatrix}
  \quad\mbox{as well as}\quad
  V=
  \begin{pmatrix}
    1 & -\frac{1}{2}\\[0.4em]
    \frac{1}{2} & -\frac{1}{6}
  \end{pmatrix}\;,
\end{equation*}
we have
\begin{equation*}
  \e^{\eps tA}\Gamma_t^\eps=
  D_\eps J_t\left(V+\eps t R+O\left(\eps^2\right)\right)J_t D_\eps\;,
\end{equation*}
uniformly in $t\in[0,1]$.
Let $I$ denote the $2\times 2$ identity matrix.
Since $V$ is invertible, we deduce that, for $\eps>0$
sufficiently small,
\begin{equation*}
  \left(V+\eps R+O\left(\eps^2\right)\right)^{-1}=
  \left(I+\eps V^{-1}R+O\left(\eps^2\right)\right)^{-1}V^{-1}=
  V^{-1}-\eps V^{-1}RV^{-1}+O\left(\eps^2\right)\;,
\end{equation*}
and therefore, due to $J_1=I$,
\begin{equation*}
  \left(\Gamma_1^\eps\right)^{-1}\e^{-\eps A}=
  D_\eps^{-1}\left(V^{-1}-
    \eps V^{-1}RV^{-1}+O\left(\eps^2\right)\right) D_\eps^{-1}\;.  
\end{equation*}
This implies that
\begin{align*}
  \alpha_t^\eps
  &=D_\eps J_t\left(V+\eps t R+O\left(\eps^2\right)\right)J_t
    \left(V^{-1}-\eps V^{-1}RV^{-1}+O\left(\eps^2\right)\right) D_\eps^{-1}\\
  &=D_\eps\left(J_tVJ_tV^{-1} + \eps t J_tRJ_tV^{-1} -
    \eps J_tV J_t V^{-1}RV^{-1} +O\left(\eps^2\right)\right) D_\eps^{-1}\;.
\end{align*}
We compute
\begin{equation*}
  J_tVJ_tV^{-1}=
  \begin{pmatrix}
    3t^2-2t & 6t-6t^2\\[0.4em]
    t^3-t^2 & 3t^2-2t^3
  \end{pmatrix}
\end{equation*}
as well as
\begin{equation*}
  \eps t J_tRJ_tV^{-1}=
  \begin{pmatrix}
    -2\eps t^3 & 4\eps t^3\\[0.4em]
    -\frac{2}{3}\eps t^3 & 2\eps t^3
  \end{pmatrix}
  \quad\mbox{and}\quad
  \eps J_tV J_t V^{-1}RV^{-1}=
  \begin{pmatrix}
    -2\eps t^2 & 4\eps t\\[0.4em]
    -\frac{2}{3}\eps t^3 & 2\eps t^2
  \end{pmatrix}\;,
\end{equation*}
which yields
\begin{equation*}
  \alpha_t^\eps=
  \begin{pmatrix}
    3t^2-2t + \left(2t^2-2t^3\right)\eps + O\left(\eps^2\right) &
    \left(6t-6t^2\right)\eps^{-1} + 4t^3-4t + O\left(\eps\right)\\[0.4em]
    \left(t^3-t^2\right)\eps + O\left(\eps^3\right) &
    3t^2-2t^3+\left(2t^3-2t^2\right)\eps + O\left(\eps^2\right)
  \end{pmatrix}\;.
\end{equation*}
In particular, for $y=(a,b)$, we obtain
\begin{equation*}
  \alpha_t^\eps y=
  \begin{pmatrix}
    3t^2-2t & \left(6t-6t^2\right)\eps^{-1}\\[0.4em]
    \left(t^3-t^2\right)\eps & 3t^2-2t^3
  \end{pmatrix}
  \begin{pmatrix}
    a \\[0.4em] b
  \end{pmatrix}+
  \begin{pmatrix}
    \left(4t^3-4t\right)b\\[0.4em] \left(2t^3-2t^2\right) b\eps
  \end{pmatrix}+
  \begin{pmatrix}
    O\left(\eps\right) \\[0.4em] O\left(\eps^2\right)
  \end{pmatrix}\;.
\end{equation*}
It follows that in our current example
for an approximate minimal-like path to
lead to well-defined small-time fluctuations
for the bridge from $x=0$ to $y=(a,b)$ 
with respect to the
rescaling $D_\eps$, we have to at least subtract the path
\begin{equation*}
  \left(
  \left(3t^2-2t\right)a + \left(6t-6t^2\right)\frac{b}{\eps} +
  \left(4t^3-4t\right)b,
  \left(t^3-t^2\right) a\eps +
  \left(3t^2-2t^3\right)b+\left(2t^3-2t^2\right)b\eps\right)_{t\in[0,1]}\;.
\end{equation*}
This differs from the minimal-like path
$(\phi_t^\eps(y))_{t\in[0,1]}$ considered for the
Kolmogorov diffusion, and the
approximate minimal-like path
$(\psi_t^\eps(y))_{t\in[0,1]}$ found
for the Ornstein-Uhlenbeck process paired with its area
starting from $0$.

%% file: genkol.tex
\subsection{Iterated Kolmogorov diffusion}
\label{ultra}
The diffusions studied in Section~\ref{sec:OU} and
Section~\ref{sec:diffmin} both
exhibit the same small-time fluctuations
for the bridge as the Kolmogorov
diffusion. Similarly, there is a family of diffusions which
all have the same small-time fluctuations for the bridge as the
iterated Kolmogorov diffusion, that is, a standard
Brownian motion together with a
finite number of its iterated time integrals. Banerjee and
Kendall~\cite{kendall_couple} study maximal and efficient couplings for
iterated Kolmogorov diffusions, and Baudoin, Gordina and
Mariano~\cite{kolgradbounds} obtain gradient bounds for this
hypoelliptic diffusion.
We close by explicitly
determining the small-time fluctuations for the
bridge of an iterated Kolmogorov diffusion.
By the independence of the components of a Brownian motion in $\R^m$,
it is sufficient to focus on a standard
Brownian motion in $\R$ and its iterated time integrals.
In our model class, this diffusion corresponds to the
choice of the $d\times d$ matrix $A$ and the $d\times 1$ matrix $B$,
understood as a column vector, whose
entries are, for $i,j\in\{1,\dots,d\}$,
\begin{equation*}
  A_{ij}=
  \begin{cases}
    1 & \mbox{if } i=j+1\\
    0 & \mbox{otherwise}
  \end{cases}
  \quad\mbox{and}\quad
  B_{i}=
  \begin{cases}
    1 & \mbox{if } i=1\\
    0 & \mbox{otherwise}
  \end{cases}\;.
\end{equation*}
With $\lo$ on $\R^d$ given by~(\ref{model}), the operator
$\lo-\frac{\pt}{\pt t}$ is a
strongly degenerate ultraparabolic operator.
For $k\in\N$, we have
\begin{equation*}
  \left(A^k\right)_{ij}=
  \begin{cases}
    1 & \mbox{if } i=j+k\\
    0 & \mbox{otherwise}
  \end{cases}\;.
\end{equation*}
This yields
\begin{equation*}
  \operatorname{span}\left\{
    B,AB,\dots,A^{d-1}B\right\}=\R^d\;,
\end{equation*}
that is, the Kalman rank condition~(\ref{kalman}) is satisfied.
Moreover, we obtain $n=d$
since $\R^d$ cannot be spanned by less than $d$ vectors.
From $A^k=0$ for $k\geq d$, it follows that, for $r\in\R$,
\begin{equation*}
  \e^{\eps rA}=
  \sum_{k=0}^{d-1}\frac{\eps^k r^k}{k!}A^k\;,
\end{equation*}
which implies
\begin{equation*}
  \left(U^\eps(r)\right)_i=\left(\e^{\eps rA}B\right)_i=
  \frac{\eps^{i-1}r^{i-1}}{(i-1)!}\;.
\end{equation*}
Hence, $U^\eps(r)$ is of the form~(\ref{urepsform}) with
\begin{equation*}
  u_i=\frac{1}{(i-1)!}\;.
\end{equation*}
In the current example, the matrices
$D_\eps$ and $J_t$ are the $d\times d$ diagonal matrices, whose
$i^{\rm th}$ diagonal element equals $\eps^{i-1}$ and $t^{i-1/2}$,
respectively. We further see that $V$ has the entries
\begin{equation*}
  V_{ij}=(-1)^{j+1}u_iu_j\frac{(i-1)!\,(j-1)!}{(i+j-1)!}=
  (-1)^{j+1}\frac{1}{(i+j-1)!}\;.
\end{equation*}
Let $H$ be the $d\times d$ Hankel matrix defined by
\begin{equation*}
  H_{ij}=\frac{1}{(i+j-1)!}\;,
\end{equation*}
and let $S$ be the $d\times d$ diagonal matrix whose
$i^{\rm th}$ diagonal element equals $(-1)^{i+1}$. Due to
\begin{equation*}
  (HS)_{ij}=\sum_{k=1}^d H_{ik}S_{kj}=H_{ij}S_{jj}=
  (-1)^{j+1}\frac{1}{(i+j-1)!}=V_{ij}\;,
\end{equation*}
for all $i,j\in\{1,\dots,d\}$,
we have $V=HS$. Since $S^{-1}=S$, it follows that
\begin{equation*}
  V^{-1}=SH^{-1}\;.
\end{equation*}
Using my formula for the
inverse of the factorial Hankel matrix $H$,
see~\cite[Theorem~1.1]{inverthankel}, we obtain
\begin{align*}
  \left(V^{-1}\right)_{ij}
  &=\sum_{k=1}^dS_{ik}\left(H^{-1}\right)_{kj}
    =(-1)^{i+1}\left(H^{-1}\right)_{ij}\\
  &=(-1)^{d+j}(i-1)!\,j!\,\binom{d-1}{i-1}\binom{d+j-1}{j}
    \sum_{k=0}^{i-1}\binom{d-i+k}{j-1}\binom{d+k-1}{k}\;.
\end{align*}
We further compute that
\begin{align*}
  &\left(J_tVJ_tV^{-1}\right)_{ij}\\
  &\qquad=\sum_{l=1}^d t^{i-1/2}V_{il}t^{l-1/2}\left(V^{-1}\right)_{lj}\\
  &\qquad=\sum_{l=1}^d(-1)^{d+j+l+1}\frac{(l-1)!\,j!}{(i+l-1)!}
    \binom{d-1}{l-1}\binom{d+j-1}{j}
    \sum_{k=0}^{l-1}\binom{d-l+k}{j-1}\binom{d+k-1}{k} t^{i+l-1}\;.
\end{align*}
As $\hat U(r)$ has the entries
\begin{equation*}
  \left(\hat U(r)\right)_i=r^{i-1}u_i=\frac{r^{i-1}}{(i-1)!}\;,
\end{equation*}
we see, by integration by parts, that the process
\begin{equation*}
  \left(\int_0^t\hat U(t-s)\dd W_s\right)_{t\in[0,1]}
\end{equation*}
is again the iterated Kolmogorov diffusion.
Using Theorem~\ref{CLT}, this observation
and the formula for $J_tVJ_t V^{-1}$ together
give an explicit expression of 
the small-time fluctuations for the bridge of an
iterated Kolmogorov diffusion. Moreover, since
$U^\eps(r)=D_\eps\hat U(r)$ for $r\in\R$, these small-time
fluctuations are equal in law to the bridge from $0$ to $0$ in time
$1$ of an iterated Kolmogorov diffusion with the same number of
iterated time integrals.

%% file: modelclass.bbl
\begin{thebibliography}{10}

\bibitem{BMN}
Ismael Bailleul, Laurent Mesnager, and James Norris.
\newblock {\it Small-time fluctuations for the bridge of a sub-Riemannian
  diffusion}.
\newblock {\tt arXiv:1505.03464}, 13 May 2015. To appear in {\it Annales
  scientifiques de l'{\'E}.N.S.}

\bibitem{kendall_couple}
Sayan Banerjee and Wilfrid~S. Kendall.
\newblock Coupling the {K}olmogorov diffusion: maximality and efficiency
  considerations.
\newblock {\em Advances in Applied Probability}, 48(A):15--35, 2016.

\bibitem{barilariandpaoli}
Davide Barilari and Elisa Paoli.
\newblock Curvature terms in small time heat kernel expansion for a model class
  of hypoelliptic {H}\"ormander operators.
\newblock {\em Nonlinear Analysis. Theory, Methods \& Applications. An
  International Multidisciplinary Journal}, 164:118--134, 2017.

\bibitem{kolgradbounds}
Fabrice Baudoin, Maria Gordina, and Phanuel Mariano.
\newblock {\it Gradient {B}ounds for {K}olmogorov {T}ype {D}iffusions}.
\newblock {\tt arXiv:1803.01436}, 4 March 2018.

\bibitem{expdecay1}
G\'erard Ben~Arous and R\'emi L\'eandre.
\newblock D\'ecroissance exponentielle du noyau de la chaleur sur la diagonale.
  {I}.
\newblock {\em Probability Theory and Related Fields}, 90(2):175--202, 1991.

\bibitem{expdecay2}
G\'erard Ben~Arous and R\'emi L\'eandre.
\newblock D\'ecroissance exponentielle du noyau de la chaleur sur la diagonale.
  {II}.
\newblock {\em Probability Theory and Related Fields}, 90(3):377--402, 1991.

\bibitem{inverthankel}
Karen Habermann.
\newblock {\it An explicit formula for the inverse of a factorial {H}ankel
  matrix}.
\newblock {\tt arXiv:1808.02880}, \\ 8~August 2018.

\bibitem{MAP}
Karen Habermann.
\newblock Small-time fluctuations for sub-{R}iemannian diffusion loops.
\newblock {\em Probability Theory and Related Fields}, 171(3-4):617--652, 2018.

\bibitem{hormander}
Lars H{\"o}rmander.
\newblock Hypoelliptic second order differential equations.
\newblock {\em Acta Mathematica}, 119:147--171, 1967.

\bibitem{kalman_book}
Rudolf~E. Kalman, Peter~L. Falb, and Michael~A. Arbib.
\newblock {\em Topics in {M}athematical {S}ystem {T}heory}.
\newblock McGraw-Hill, New York, 1969.

\bibitem{kolmogorov}
Andrey Kolmogoroff.
\newblock Zuf\"allige {B}ewegungen (zur {T}heorie der {B}rownschen {B}ewegung).
\newblock {\em Annals of Mathematics. Second Series}, 35(1):116--117, 1934.

\bibitem{gaussian}
John Lamperti.
\newblock On {L}imit {T}heorems for {G}aussian {P}rocesses.
\newblock {\em Annals of Mathematical Statistics}, 36:304--310, 1965.

\bibitem{evolution}
Ermanno Lanconelli and Sergio Polidoro.
\newblock On a class of hypoelliptic evolution operators.
\newblock {\em Universit\`a e Politecnico di Torino. Seminario Matematico.
  Rendiconti}, 52(1):29--63, 1994.
\newblock Partial differential equations, II (Turin, 1993).

\end{thebibliography}
